\documentclass[12pt, reqno]{amsart}
\usepackage{amsmath, amsthm, amscd, amsfonts, amssymb, graphicx, color}
\usepackage[bookmarksnumbered, colorlinks, plainpages]{hyperref}
\input{mathrsfs.sty}

\makeatletter
\@namedef{subjclassname@2010}{%
  \textup{2010} Mathematics Subject Classification}
\makeatother

\newtheorem{thm}{Theorem}[section]
\newtheorem{cor}[thm]{Corollary}
\newtheorem{lem}[thm]{Lemma}

\theoremstyle{definition}

\newtheorem{rem}[thm]{Remark}

\numberwithin{equation}{section}

\begin{document}

\baselineskip=17pt

\title[Deformation of involution and multiplication in a $C^*$-algebra]{Deformation of involution and multiplication in a $C^*$-algebra}

\author[H. Najafi, M.S. Moslehian]{H. Najafi and M. S. Moslehian}
\address{Department of Pure Mathematics, Center of Excellence in
Analysis on Algebraic Structures (CEAAS), Ferdowsi University of
Mashhad, P. O. Box 1159, Mashhad 91775, Iran}
\email{hamednajafi20@gmail.com} \email{moslehian@um.ac.ir and
moslehian@member.ams.org}

\date{}

\begin{abstract}
We investigate the deformation of involution and multiplication in a
unital $C^*$-algebra when its norm is fixed. Our main result is to
present all multiplications and involutions on a given $C^*$-algebra
$\mathcal{A}$ under which $\mathcal{A}$ is still a $C^*$-algebra
whereas we keep the norm unchanged. For each invertible element
$a\in\mathcal{A}$ we also introduce an involution and a
multiplication making $\mathcal{A}$ into a $C^*$-algebra in which
$a$ becomes a positive element. Further, we give a necessary and
sufficient condition for that the center of a unital $C^*$-algebra
$\mathcal{A}$ is trivial.
\end{abstract}



\maketitle

\section{Introduction}

A $C^*$-algebra is a complex Banach $*$-algebra $\mathcal{A}$
satisfying $ \|a^*a\|= \|a\|^2$ $(a\in \mathcal{A})$. By the
Gelfand--Naimark theorem, a $C^*$-algebra is a norm closed
$*$-subalgebra of $\mathbb{B}(\mathcal{H})$ for some Hilbert space
$\mathcal{H}$. A strongly closed $*$-subalgebra of
$\mathbb{B}(\mathcal{H})$ containing the identity operator is called
a von Neumann algebra. By the double commutant theorem a unital
$*$-subalgebra $\mathcal{A}$ of $\mathbb{B}(\mathcal{H})$ is a von
Neumann algebra if and only if $\mathcal{A}$ is equal to its double
commutant $\mathcal{A}^{cc}$, where $\mathcal{A}^{c}=\{B \in
\mathbb{B}(\mathcal{H}): AB=BA \mbox{~for all~} A\in \mathcal{A}\}$.
By Sakai's characterization of von Neumann algebras, $\mathcal{A}$
is a von Neumann algebra if and only if it is a $W^*$-algebra, i.e.
it is a $C^*$-algebra being the norm dual of a Banach space
$\mathcal{A}_*$. Throughout the paper $\mathcal{A}$ denotes an
arbitrary $C^*$-algebra and $\mathcal{Z}(\mathcal{A})$ stands for its center. \\
For a self adjoint element $a\in \mathcal{A}$, it holds that $r(a)=
\|a\|$, where $r(a)$ denotes the spectral radius of $\mathcal{A}$.
This implies that the norm of a $C^*$-algebra is unique when we fix
the involution and the multiplication. Indeed, if $\mathcal{A}$ is a
$C^*$-algebra under two norms $\|.\|_1$ and $\|.\|_2$, then
$\|a\|_1=\|a^*a\|_1^{\frac{1}{2}} = r(a^*a)^{\frac{1}{2}}=
\|a^*a\|_2^{\frac{1}{2}} = \|a\|_2$
for all $a\in \mathcal{A}$.\\
Bohnenblust and Karlin \cite{B-K} showed that there is at most one
involution on a Banach algebra with the unit $1$ making it into a
$C^*$-algebra (see also \cite{RIC}): Let $*$ and $\#$ be two
involutions on a unital Banach algebra $\mathcal{A}$ making it into
$C^*$-algebras. Let $x \in \mathcal{A}$. It follows from the fact
``an element $x$ of a unital $C^*$-algebra is self-adjoint if and
only if $\tau(x)$ is real for every bounded linear functional $\tau$
on $\mathcal{A}$ with $\parallel \tau
\parallel =\tau(1)=1$ (\cite[Proposition 4.3.3]{Ka-Ri1})'' that $x$ is self-adjoint with respect to $*$ if and
only if it is self-adjoint with respect to $\#$. Now let $a \in
\mathcal{A}$ be arbitrary and $a=a_1+{\rm i}a_2$ with self-adjoint
parts $a_1, a_2$ with respect to $*$. Then $a_1^*=a_1^\#$ and
$a_2^*=a_2^\#$ and $a^*=a_1-{\rm i}a_2=a^\#$. There is another way
to show the uniqueness of the involution. Indeed if $*$ and $\#$ be
two involutions on a unital Banach algebra $\mathcal{A}$ making it
into $C^*$-algebras, then the identity map from $(\mathcal{A}, *)$
onto $(\mathcal{A},\#)$ is positive (see \cite[Proposition
2.11]{PAU}) and so $a^{*}=a^{\#}$ for all $a\in \mathcal{A}$.

There are several characterizations of $C^*$-algebras among
involutive Banach algebras, see \cite{TOM} in which the authors
start with a $C^*$-algebra and modify its structure. We however
investigate a different problem in the same setting. In fact we
investigate the deformation of involution and multiplication in a
unital $C^*$-algebra when its norm is fixed. Our main result is to
present all multiplications $\circ$ and involutions $\star$ on a
given $C^*$-algebra $\mathcal{A}$ under which $\mathcal{A}$ is still
a $C^*$-algebra whereas we keep the norm unchanged. As an
application, for each invertible element $a\in\mathcal{A}$ we
introduce an involution and a multiplication making $\mathcal{A}$
into a $C^*$-algebra in which $a$ becomes a positive element.
Further, we give a necessary and sufficient condition for that the
center of a unital $C^*$-algebra $\mathcal{A}$ is trivial.

Recall that a Jordan $*$-homomorphism is a self-adjoint map
preserving squares of self-adjoint operators. Jacobson and Rickart
\cite{Ja} showed that for every Jordan $*$-homomorphism $\rho$ of a
 $C^*$-algebra $\mathcal{A}$ with the unit $1$ into a von Neumann algebra
$\mathcal{B}$ there exist central projections $p_1, p_2\in
\mathcal{B}$ such that $\rho=\rho_1+\rho_2$, $\rho(1)=p_1+p_2$,
$\rho_1(a)=\rho(a)p_1$ is a $*$-homomorphism and
$\rho_2(a)=\rho(a)p_2$ is a $*$-antihomomorphism. Kadison \cite{Ka}
showed that an isometry of a unital $C^*$-algebra onto another
$C^*$-algebra is a Jordan $*$-isomorphism.

\section{Results}
We start our work with the following lemma.

\begin{lem}\label{1-1}
Let $\mathcal{A}$ be a unital $C^*$-algebra of operators acting on a
Hilbert space $\mathcal{H}$. Let $p\in \mathcal{A}$ be a central
projection and $u\in \mathcal{A}$ be a unitary. Let $\circ$ be the
multiplication and $\star$ be the involution defined on
$\mathcal{A}$ by
\begin{align}\label{11}
a\circ b=p a u b + (1-p) b u a \quad{\rm ~ and ~}\quad a^\star&=u^*
a^* u^*
\end{align}
for $a,b \in \mathcal{A}$, respectively. Then $\mathcal{A}$ equipped
with the multiplication $\circ$ and the involution $\star$ is a
unital $C^*$-algebra.
\end{lem}

\begin{proof}
It is easy to check that $\mathcal{A}$ is a complex Banach algebra
under the multiplication $\circ$ and $u^*$ is the unit for this
multiplication. By the decomposition $\mathcal{H}=
p\mathcal{H}\oplus(1-p)\mathcal{H}$, we can represent any element $a
\in\mathcal{A}$ by the $2\times 2$ matrix {\scriptsize $ \left(
 \begin{array}{cc}
  pa & 0 \\
  0 & (1-p)a \\
 \end{array}
\right).$} For $a,b \in \mathcal{A}$, therefore $pa+(1-p)b$ can be
identified by {\scriptsize $\left(
 \begin{array}{cc}
  pa & 0 \\
  0 & (1-p)b \\
 \end{array}
\right)$}, whence $||pa+(1-p)b||=\max(\|p a \|, \|(1-p)b\|)$. Hence
 \begin{eqnarray*}
\|a^\star \circ a\|&=& \|p\ u^* a^* a + (1-p) \ a a^* u^* \|\\
&=&\max(\|p a^* a \|, \| a a^* (1-p)\|)\\
&=&\max(\|p a \|^2, \|(1-p)a \|^2)\\
&=&\max(\|p a \|, \|(1-p)a\|)^2 =\|a\|^2
\end{eqnarray*}
for all $a \in \mathcal{A}$.
\end{proof}

The unital $C^*$-algebra $\mathcal{A}$ equipped with the
multiplication $\circ$ and the involution $\star$ is denoted by
$\mathcal{A}(\circ, \star)$. Next we establish a converse of Lemma
\ref{1-1}.
%===================================================================================================================

\begin{thm}\label{12}
Let $\mathcal{A}$ be a unital $C^*$-algebra of operators acting on a
Hilbert space $\mathcal{H}$ and there exist a multiplication $\circ$
and an involution $\star$ on the normed space $\mathcal{A}$ making
it into a $C^*$-algebra.
 Then there exists a unitary element $u\in \mathcal{A}$ and a central projection $p$ in the double commutant
 of $\mathcal{A}^{cc}$ of $\mathcal{A}$ such that both equalities
\eqref{11} hold.
\end{thm}
\begin{proof}
Since $\mathcal{A}$ is unital, the closed unit ball of $\mathcal{A}$
has an extreme point,
 hence the $C^*$-algebra $\mathcal{A}(\circ,\star)$ is unital. Since $\iota(x)=x$ is an isometric linear map of
$\mathcal{A}$ onto $\mathcal{A}(\circ,\star)$, the unitary elements
of $\mathcal{A}(\circ,\star)$ and those of $\mathcal{A}$ coincide
\cite[Exercise 7.6.17]{Ka-Ri}. Thus if $u^*$ is the unit of
$\mathcal{A}(\circ,\star)$, then $u$ is a unitary of $\mathcal{A}$.
Define $\rho:\mathcal{A}\rightarrow \mathcal{A}(\circ,\star)$ by
$\rho(a)=u^*a$. Clearly $\rho$ is a unital isometric linear map of
$\mathcal{A}$ onto $\mathcal{A}(\circ,\star)$. Hence $\rho$ is a
positive map. This implies that $u^*a^*=\rho(a^*)=(u^*a)^{\star}$
and so $a^{\star}= u^*a^*u^*$.\\
For determining the multiplication, define a multiplication
$\diamond$ on $\mathcal{A}^{cc}$ (with respect to the original
multiplication) by \eqref{11} as $p=1$.
 Then $\mathcal{A}^{cc}$ with the multiplication $\diamond$ is a $C^*$-algebra. The space $\mathcal{A}^{cc}$ as a Banach space is already the dual of a
 Banach space, so $\mathcal{A}^{cc}$ with the new product and the new involution is a von Neumann algebra. Then the map $\rho(x)=x$ is a unital isometric linear map
 of $\mathcal{A}(\circ,\star)$ into the von Neumann algebra $\mathcal{A}^{cc}(\diamond,\star)$.
  By the result of Kadison \cite{Ka} it is a Jordan $*$-isomorphism and by the Jacobson and Rickart theorem \cite{Ja}
  there exists a central projection $p^{'}$ in $\mathcal{A}^{cc}(\diamond,\star)$
   such that $\rho_1(x)= p^{'} \diamond \rho(x)$ is a
$*$-homomorphism and $\rho_2(x)= (u^*-p^{'})\diamond \rho(x)$ is a
$*$-antihomomorphism. Therefore for each $a,b\in \mathcal{A}$ we
have
 \begin{align*}
 a\circ b&=\rho(a\circ b)\\
 &= \rho_1(a\circ b)+\rho_2(a\circ b)\\
 &=p^{'} \diamond \rho_1(a) \diamond \rho_1(b) +(u^*-p^{'})\diamond \rho_2(b) \diamond \rho_2(a)\\
 &=p^{'} \diamond a \diamond b +(u^*-p^{'})\diamond b \diamond a.\\
 &=p^{'} u a u b +(u^*-p^{'})u b u a\\
 &=p^{'} u a u b +(1-p^{'}u) b u a.
\end{align*}
 Let $p=p^{'}u$. Since $(p^{'}u)^2= p^{'}u p^{'}u 1= p^{'} \diamond (p^{'} \diamond 1)=p^{'}
\diamond 1=p^{'}u$ and $(p^{'}u)^* = u^* p^{'*}= u^* p^{'*} u^* u=
p^{'\star}u = p^{'}u$,
 so $p$ is a projection in $\mathcal{A}^{cc}$.
 A similar argument shows that $\theta:\mathcal{A}^{cc}(\diamond,\star)\rightarrow \mathcal{A}^{cc}$
 defined by $\theta(a)=a u$ is a Jordan $*$-isomorphism. So,
 by \cite[Corollary 1]{Ja} , $\theta(\mathcal{Z}(\mathcal{A}^{cc}(\diamond,\star)))=\mathcal{Z}(\theta(\mathcal{A}^{cc}(\diamond,\star)))$.
 Therefore $pa=\theta(p^{'})\theta(au^*)=\theta(au^*)\theta(p^{'})=ap$ for each $a\in \mathcal{A}$.
 Hence $p$ is a central projection in $\mathcal{A}^{cc}$.
\end{proof}
%===================================================================================================================
\begin{rem}
Note that in general case, a $C^*$-algebra $\mathcal{A}$ has many
representations. However the proof of Theorem \ref{12} shows that
for any representation of $\mathcal{A}$, we can present all
multiplications and involutions on $\mathcal{A}$ which keep it still
a C*-algebra with the same norm by a unitary and a central
projection in the double commutant with respect to the same
representation. Further, since $p$ in Theorem \ref{12} is in
$\mathcal{A}^{cc}\subseteq \mathbb{B}(\mathcal{H})$, it depends on
$\mathcal{H}$. If $\mathcal{A}$ is a von Neumann algebra, then $p
\in \mathcal{A}^{cc}=\mathcal{A}$.
\end{rem}
%===================================================================================================================
 \begin{cor}
Let $\mathcal{I}$ be an ideal of a von Neumann algebra
$\mathcal{A}$. Then $\mathcal{I}$ is also an ideal of the
$C^*$-algebra $\mathcal{A}(\circ,\star)$ for any multiplication
$\circ$ and any involution $\star$.
 \end{cor}
\begin{proof}
It is sufficient to note that $p a u b$ and $(1-p) b u a$ belong to
$\mathcal{I}$ when $a\in \mathcal{A}, b\in \mathcal{I}$ and so
 $a\circ b=p a u b + (1-p) b u a\in \mathcal{I}$.
\end{proof}
 %===================================================================================================================

It is easy to see that $a \circ b = b \circ a$ if and only if
$aub=bua$. We therefore have
 %==================================================================================================================
 \begin{cor}\label{13}
Suppose that $\mathcal{A}$ is a unital $C^*$-algebra and the normed
space $\mathcal{A}$ is equipped with a multiplication $\circ$
 and an involution $\star$ is a $C^*$-algebra with the unit $u^*$, where $u\in \mathcal{A}$ is a unitary. Then\\
 (i) $\mathcal{A}$ is commutative if and only if so is $\mathcal{A}(\circ,\star)$.\\
 (ii) $\mathcal{Z}(\mathcal{A})= \mathbb{C}1$ if and only if $\mathcal{Z}(\mathcal{A}(\circ,\star))= \mathbb{C}u^*$.
  \end{cor}
 \begin{proof}
(i) Let $\mathcal{A}$ be commutative. By Theorem \ref{12} there
exist a unitary element $u\in \mathcal{A}$ and a central projection
$p$
  in $\mathcal{A}^{cc}$ such that
 \begin{align*}
a\circ b&=p a u b + (1-p) b u a \ \ \ \ \ \ \ \  \  \ (a,b\in
\mathcal{A} ).
 \end{align*}
Hence
 \begin{align*}
a \circ b=p aub+(1-p)bua = p bua + (1-p)aub= b\circ a\,.
 \end{align*}
Therefore $\mathcal{A}(\circ,\star)$ is commutative. Changing the role of $\mathcal{A}$ by $\mathcal{A}(\circ,\star)$, we reach the reverse assertion.\\
(ii) Let $\mathcal{Z}(\mathcal{A})= \mathbb{C}1$. If $a\in
\mathcal{Z}(\mathcal{A}(\circ,\star))$, then for any $b\in
\mathcal{A}$ we have $a \circ b= b\circ a$. As in the proof of
Theorem \ref{12} we observe that
$\theta:\mathcal{A}^{cc}(\circ,\star)\rightarrow \mathcal{A}^{cc}$
defined by $\theta(a)=a u$ is a Jordan $*$-isomorphism. Hence, by
\cite[Corollary 1]{Ja} , $\theta(b)\theta(a)=\theta(a)\theta(b)$, so
$aubu=buau$. Since each element of $\mathcal{A}$ is of the form $bu$
for some $b\in \mathcal{A}$, it follows that $au\in
\mathcal{Z}(\mathcal{A})$. Hence $au=\lambda1$ for some $\lambda \in
\mathbb{C}$. Therefore $\mathcal{Z}(\mathcal{A}(\circ,\star))=
\mathbb{C}u^*$. Similarly we can deduce the converse.
  \end{proof}
%=======================================================================================================
\begin{rem}
The Arens product on $(c_0)^{**}=l^{\infty}$ coincide with the usual
product in $l^{\infty}$ \cite[Example 2.6.22]{Da}. It was extended
to arbitrary $C^*$-algebras in \cite{BD}. We reprove the fact in our
own way: Let $\mathcal{A}$ be a $C^*$-algebra and its second dual
$\mathcal{A^{**}}$ be also a $C^*$-algebra under a multiplication
$(a,b)\longmapsto a\cdot b$ whose restriction to $\mathcal{A}\times
\mathcal{A}$ is the same multiplication of $\mathcal{A}$. We shall
show that the Arens product (denoted by $\diamond$) on
$\mathcal{A^{**}}$ is the same as the multiplication $\cdot$ on
$\mathcal{A^{**}}$. It is known that $\mathcal{A^{**}}$ is a von
Neumann algebra under the Arens multiplication \cite[Theorem
3.2.37]{Da}. By the Kaplansky density theorem, $\mathcal{A}$ is
dense in $\mathcal{A^{**}}$ in the weak$^*$-topology, so there
exists a net $u_{\alpha}$ in $\mathcal{A}$ such that
$u_{\alpha}\rightarrow 1$ in the weak$^*$-topology in which $1$
denotes the unit of $\mathcal{A^{**}}$. So
\begin{align*}
b= w^*-\lim_{\alpha}u_{\alpha}b= w^*-\lim_{\alpha} u_{\alpha}
\diamond b= 1 \diamond b
\end{align*}
for each $b\in \mathcal{A}$. The Kaplansky density theorem implies
that $1 \diamond x=x$ for each $x \in \mathcal{A^{**}}$. Therefore
the units of both multiplications $\cdot$ and $\diamond$ are same.
By Theorem \ref{12} there exist a central projection $p\in
\mathcal{A}$ such that
\begin{align*}
x \diamond y =p x y + (1-p) y x,
\end{align*}
for each $x,y\in \mathcal{A^{**}}$. On the other hand for each
$a,b\in \mathcal{A}$, we have $a\diamond b = ab$. So
$(1-p)ab=(1-p)ba$. Since $\mathcal{A}$ is dense in
$\mathcal{A^{**}}$ in the weak$^*$-topology, we have
$(1-p)xy=(1-p)yx$ for each $x,y \in \mathcal{A^{**}}$. Therefore $x
\diamond y =pxy + (1-p) y x= p x y + (1-p)xy = x y$ for each $x,y
\in \mathcal{A^{**}}$. For instance, we deduce that the Arens
product on $\mathbb{K}(\mathcal{H})^{**}=\mathbb{B}(\mathcal{H})$ is
equal to the operator multiplication on $\mathbb{B}(\mathcal{H})$.
\end{rem}

%====================================================================================================================
 \begin{thm}
 Let $\mathcal{A}$ be a unital $C^*$-algebra. Then the following assertions are equivalent:\\
 (i) $\mathcal{Z}(\mathcal{A})= \mathbb{C}1$\\
 (ii) If for invertible operators $a,b\in \mathcal{A}$, $ ||axb||=||x||$ holds for each $x\in \mathcal{A}$,
  then there exists $\lambda > 0$ such that both $\lambda a$ and $\frac{1}{\lambda} b$ are unitary.
 \end{thm}
 \begin{proof}
 ${\rm (i)} \Rightarrow {\rm (ii)}$ Note that if $||a^{-1}xa||\leq ||x||$ for each $x\in \mathcal{A}$,
 then map $\varphi(x)=a^{-1}xa$
 is a contractive unital linear map on $\mathcal{A}$ .
 It follows from \cite[Proposition 2.11]{PAU} that $\varphi$ is positive.
 Therefore $(a^{-1}xa)^*=a^{-1}x^{*}a$ and so $aa^{*} x^*=x^* a a^*$ for each $x\in \mathcal{A}$.
 Hence $aa^* \in \mathcal{Z}(\mathcal{A})=\mathbb{C}1$. So $a^*a=\lambda 1$ for some $\lambda >0$.
 Therefore $\frac{1}{\sqrt{\lambda}} a$ is unitary. First, assume that $||a x b|| = ||x||$ for positive invertible
  operators $a,b$ and each $x\in \mathcal{A}$. Then $||b^{-1} a^{-1}||=||a^{-1}b^{-1}||= ||a a^{-1}b^{-1}b||=1$,
  whence
 $$||a^{-1} x a || \leq ||a x b || \ || b^{-1} a^{-1}|| \leq ||x||.$$
 Therefore there exists $\lambda > 0$ such that $\frac{1}{\lambda} a$ is unitary.
 Since $\frac{1}{\lambda} a$ is positive and unitary we have $a=\lambda$.
 A similar argument shows that $b=\lambda^{'}$.
 It follows from $1=||1||=||ab||=\lambda^{'}\lambda$ that $\lambda= \frac{1}{\lambda^{'}}$.
 Second, assume that $||axb||=||x||$ for invertible operators $a,b$ and each $x\in \mathcal{A}$.
  Utilizing the polar decompositions of $a$ and $b^*$, there exist unitary operators $u,v$
  such that $a=u|a|$ and $b=|b^*|v$. Hence $|| \ |a| \ x \ |b^*| \ ||=||u\ |a| \ x \ |b^*| \ v ||=||axb||=||x||$
  for each $x\in \mathcal{A}$. The above argument shows that $|a|=\lambda$ and $|b^*|=\frac{1}{\lambda}$
  for some $\lambda>0$, so $a= \lambda u$ and $b= \frac{1}{\lambda} v$.\\
 ${\rm (ii)} \Rightarrow {\rm (i)}$ Note that each central invertible element $a$ of $\mathcal{A}$ is a scalar multiple
 of a unitary element. In fact, we have $||a^{-1}xa||=||a^{-1}ax||=||x||$ for all $x\in \mathcal{A}$,
 so $\lambda a$ is unitary for some $\lambda>0$. Let $a\in \mathcal{Z}(\mathcal{A})$ be a positive element
 and $\lambda_1,\lambda_2 \in {\rm sp}(a)$ are distinct. Then there exists an invertible continuous
 function $f$ on ${\rm sp}(a)$ such that $f(\lambda_1)= \frac{1}{2}$ and
 $f(\lambda_2)=1$.
 Hence $f(a)$, which is a central invertible element should be a scalar multiple of a unitary. On the other
 hand, $\frac{1}{2},1 \in {\rm{sp}}(f(a))$, which is impossible. Hence the spectrum of $a$ is singleton, so $a=||a||$.
 Since $\mathcal{Z}(\mathcal{A})$ is a $C^*$-algebra, any one of its elements is a linear combination of four positive elements.
 Therefore $\mathcal{Z}(\mathcal{A})= \mathbb{C}1$.
 \end{proof}

%==================================================================================================================
Let $\mathcal{A}(u,p)$ denote the $C^*$-algebra given via Lemma
\ref{1-1} corresponding to a unitary $u$ and a central projection
$p$ in $\mathcal{A}$. The self-adjoint elements of
$\mathcal{A}(u,p)$ are all elements $a$ such that $au=u^*a^*$, a
fact which is independent of the choice of $p$. Also a self-adjoint
element $a$ is positive in $\mathcal{A}(u,p)$ if and only if
$a=b\circ b=pbub+(1-p)bub=bub$ for some self-adjoint element $b \in
\mathcal{A}(u,p)$ and this occurs if and only if $a$ is positive in
$\mathcal{A}(u,1)$.

\begin{thm}\label{15}
Let $\mathcal{A}$ be a $C^*$-algebra and $a\in \mathcal{A}$ be
invertible. Then there exists a
 unique unitary $u\in \mathcal{A}$ such that $a$ is a positive element of the $C^*$-algebra $\mathcal{A}(u^*,p)$
 for any central projection $p\in\mathcal{A}$.
\end{thm}

\begin{proof}
Let $a=u|a|$ be the polar decomposition of $a$. Then $u=a |a|^{-1}
\in \mathcal{A}$.
 So $a= u |a|^{\frac{1}{2}}|a|^{\frac{1}{2}}= |a|^{\frac{1}{2}\star}\circ |a|^{\frac{1}{2}}$,
 where $\circ$ is defined in $\mathcal{A}(u^*,1)$ by \eqref{11}.
So $a$ is positive in $\mathcal{A}(u^*,p)$ for every central
projection $p\in\mathcal{A}$.
 To see the uniqueness, note that if $a$ is invertible and $a, wa$ are positive for a unitary $w$, then
$a = w^{*} (wa)$. By the uniqueness of polar decomposition, we have
$w=1$. Now if $a$ is positive in $\mathcal{A}(v^*,1)$, then
$a=b^{\star}\circ b= v b^* b$. Hence $v^*u|a|=v^*a = b^* b$ is
positive. Therefore $v^*u|a|$ and $|a|$ are positive and so $v=u$
according to what we just proved.
\end{proof}
%=========================================================================================================
\begin{rem}
 The invertibility condition in Proposition \ref{15} is essential. For example let
 $\mathcal{A}=C[-1,1]$ and $f(t)=t$. If $f$ is positive in $C[-1,1](u,1)$ for a unitary function $u$, then
 there exist $g\in C[-1,1]$ such that $t=f(t)=u(t)|g(t)|^2$ for each $t\in [-1,1]$. So
 $u(t)=1$ for each $t\in (0,1]$ and $u(t)=-1$ for each $t\in [-1,0)$, which is impossible.
\end{rem}

\subsection*{Acknowledgements}
The authors would like to sincerely thank Professor Marcel de Jeu and Professor Jun Tomiyama for their
valuable comments.

\end{document}